\documentclass[letterpaper, 10 pt, conference]{ieeeconf}  

\IEEEoverridecommandlockouts                              

\usepackage{amsmath}
\usepackage{amssymb}

\usepackage{mathtools}
\usepackage{float}
\usepackage{setspace}
\usepackage[dvipsnames]{xcolor}
\usepackage{caption}
\usepackage{subcaption}
\usepackage{braket}
\usepackage{algorithm}
\usepackage{algpseudocode}
\let\labelindent\relax
\usepackage{enumitem}

\usepackage{amsthm}

\newtheoremstyle{ieee}
  {3pt}   
  {3pt}   
  {} 
  {}      
  {\bfseries} 
  {}     
  { }     
  {%
    \thmname{#1}\ \thmnumber{#2}%
    \thmnote{\ \textit{(#3)}\hspace{0pt}} 
  }

\theoremstyle{ieee}

\newtheorem{theorem}{Theorem} 
 \newtheorem{assumption}{Assumption} 
 \newtheorem{proposition}{Proposition} 
  
 \newtheorem{definition}{Definition}[section] 
 \newtheorem{remark}{Remark}[section] 
 \newtheorem{lemma}{Lemma} 
 
\newtheorem{problem}{Problem}

 \usepackage{bm} 
\usepackage{multirow}
\usepackage{adjustbox}

\newcommand*{\QEDB}{\hfill\ensuremath{\square}}

\usepackage{graphicx}

\usepackage{cite}

\newcommand{\integnneg}{\mathbb{Z}_{\geq 0}}
\newcommand{\Rank}{\operatorname{Rank}}
\definecolor{lightblue}{RGB}{90,170,255}
\definecolor{myblue}{RGB}{0,90,160}
\usepackage[
    colorlinks=true,
    linkcolor=myblue,
    citecolor=myblue,
    urlcolor=myblue
]{hyperref}

\title{\LARGE \bf
Data Poisoning Attacks Can Systematically Destabilize\\ 
Data-Driven Control Synthesis
}

\author{Vijayanand Digge$^{1}$, Martina Vanelli$^{1}$, 
Ahmad W. Al-Dabbagh$^{2}$, Julien M. Hendrickx$^{1}$, and Gianluca Bianchin$^{1}$ 
\thanks{*V. Digge 
is a FRIA grantee of the Fonds de la Recherche Scientifique – FNRS
(F.R.S.-FNRS). 
This work was supported by the Concerted Research Action (ARC) via the
“SIDDARTA” project 
and by FNRS via the “InterpoControl” Research Project.}
\thanks{$^{1}$The authors are with ICTEAM Institute and the Department
of Mathematical Engineering at UCLouvain, Belgium
        {\tt\small firstname.lastname@uclouvain.be}}%
\thanks{$^{2}$The author is with the School of Engineering, The University of British
Columbia, Kelowna, BC V1V 1V7, Canada 
        {\tt\small ahmad.aldabbagh@ubc.ca}}%
}

\begin{document}

\maketitle
\thispagestyle{empty}
\pagestyle{empty}

\begin{abstract}

Data-driven control has emerged as a powerful paradigm for synthesizing controllers directly from data, bypassing explicit model identification. However, this reliance on data introduces new and largely unexplored vulnerabilities. In this paper, we show that an attacker can systematically poison the data used for control synthesis, causing any linear state-feedback controller synthesized by the planner to destabilize the physical system.
Concerningly, we show that the attacker can achieve this objective without knowledge of the system model or the controller synthesis procedure. To this end, we develop a recursive data-poisoning mechanism that generates falsified state trajectories, inducing a precise geometric shift in the apparent system dynamics. More broadly, our results establish that data-driven control pipelines can be deterministically destabilized by model-agnostic attacks operating solely at the data level. Numerical simulations corroborate these findings for both noise-free and noisy data.

\end{abstract}

\section{Introduction}

Historically, the design and analysis of complex cyber-physical systems, such as power grids, have relied on the availability of highly accurate mathematical models. In classical control, these models can be obtained through system identification~\cite{LL:99}, which serves as a critical intermediate step between data collection and controller synthesis. This modeling phase not only enables control design, but also provides a natural point for validation and anomaly detection, where inconsistencies in the data can be identified. 
However, as modern systems grow in scale and complexity, obtaining accurate models is becoming increasingly impractical. In response, direct data-driven control has emerged as a powerful alternative paradigm~\cite{de2019formulas, van2020data, baggio2021data, markovsky2021behavioral}, which can bypass the model identification step, and instead, synthesizes feedback controllers directly from data.

The reliance on data, rather than models (which often not 
only provide parametric descriptions, but also physical 
interpretations of the system in terms of \textit{states}) 
introduces new and largely unexplored attack surfaces. In this paper, we consider an adversary/attacker that interferes with data-driven control synthesis by poisoning the data used by the system planner
as depicted in Fig.~\ref{fig:attack_mechanism}. We show that, without knowledge of the controller synthesis procedure, an attacker can systematically (i.e., deterministically) manipulate the data so that the resulting controller destabilizes the physical system. 

\begin{figure}[t]
    \centering
    \includegraphics[width=1.0\linewidth]{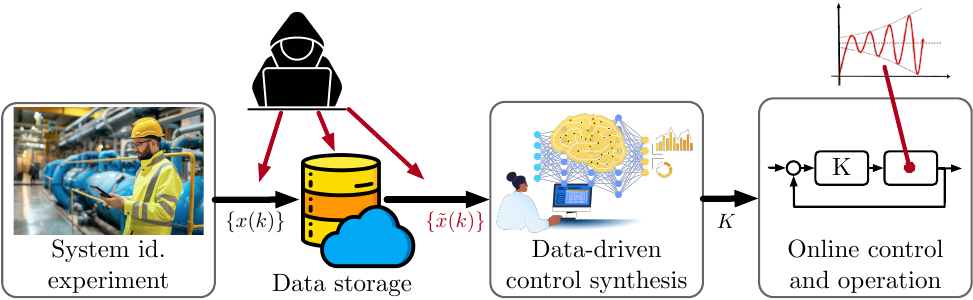}
    \caption{\small Illustration of data-poisoning attack considered in this paper. Data collected from a system identification experiment is stored and subsequently used for controller synthesis. An attacker intercepts the stored data and replaces the true state trajectory $\{x(k)\}$ with a poisoned trajectory $\{\tilde{x}(k)\}$. The resulting corrupted dataset leads the system planner to synthesize a controller $K$ that destabilizes the physical system during online operation.}
    \label{fig:attack_mechanism}
    \vspace{-0.55cm}
\end{figure}

\textit{Related works.}
In the model-based paradigm, a rich body of literature has 
investigated attack detection and 
mitigation in the line of work on \textit{cyber-physical security} (see, e.g., the non-exhaustive list of references~\cite{chen2016dynamic, ye2019summation, pasqualetti2013attack,mo2015performance,deng2016false}). 
These approaches fundamentally depend on explicit knowledge 
of system dynamics to construct observer/residual 
generators, and input-output analysis including system with 
partial observations under arbitrary attacks 
\cite{kaminaga2025data}. 
A central assumption is that the adversary exploits system vulnerabilities with some degree of model knowledge, while the defender leverages the identified model to detect deviations from expected behavior. 
Particularly related to our work is the framework of
false data injection \cite{deng2016false}, which traditionally has focused on adversarial attacks in an online setting, where sensors or actuators are compromised during system operation. 

More recently, research efforts have investigated data-driven settings in which adversaries manipulate datasets offline, prior to controller synthesis~\cite{alisic2021data, sasahara2023, fotiadis2025deception}, or target data predictive control frameworks~\cite{yu2023poisoning}. Such data poisoning attacks are fundamentally novel with respect to classical cyber-physical security frameworks: they are executed directly on the data matrices, rendering them strictly harder to detect, and they indirectly govern online system operation by structurally altering the synthesized controllers. The deployment of the resulting compromised controller destabilizes the physical system, compromising the operational reliability of data-driven control synthesis~\cite{dibaji2019}. While relevant research has explored resilience mechanisms against such poisoning~\cite{mao2022decentralized}, key limitations persist in the formulation of existing attack frameworks (e.g.,~\cite{sasahara2023,russo2021poisoning}). Specifically, these frameworks: (i) assume explicit knowledge of the underlying system model as well as of the planner's controller synthesis methodology, (ii) corrupt both the state and control input data sequences simultaneously, and (iii) lack rigorous theoretical guarantees or rely entirely on iterative, online algorithms.

\textit{Contributions.}
The main contributions of this work are twofold. 
First, we provide a rigorous and constructive characterization 
to construct data poisoning attacks offline, guaranteeing 
that the closed-loop dynamics are destabilized by the 
control synthesis, revealing a fundamental vulnerability of 
data-driven control synthesis. Interestingly, we show that this goal can be achieved by an attacker with no precise knowledge of how the controller is synthesized; to the best of our knowledge, this is shown here for the first time in the literature. 
Second, we provide a recursive poisoning mechanism to generate poisoned state trajectories, which induces a controlled geometric shift in the apparent system dynamics. Crucially, the resulting falsified data remain trajectory-consistent,
making the attack difficult to detect.
Taken together, these results establish that data-driven control pipelines can be deterministically destabilized by model-agnostic attacks that operate solely at the data level.

The paper is organized as follows. 
Section~\ref{sec:preliminaries}  provides some preliminaries.
Section~\ref{sec:problem} introduces the
problem formulation. 
Section~\ref{sec:main_results} outlines the attack scenarios
that leads to destabilization while Section~\ref{sec:simulation} presents simulation results.

\section{Preliminaries} \label{sec:preliminaries}

\subsection{Notation} 
We denote by $\mathbb{R}$, $\mathbb{Z}$, and $\mathbb{C}$ the sets of real, 
integer, and complex numbers, respectively. $\mathbb{Z}_{\geq0}$ denotes the 
set of non-negative integers.
The Euclidean vector norm and the induced matrix norm are denoted by $\|\cdot\|$. The identity matrix of appropriate dimension is denoted by $I$. For a symmetric matrix $P$, the notation $P \succ 0$ indicates positive definiteness. 
For $A \in \mathbb{R}^{n \times n}$, we denote by $\sigma(A)$ the set of eigenvalues of $A$. 
The spectral radius of $A$ is defined as 
$\rho(A) := \max_{i} |\lambda_i(A)|,$
where $\{\lambda_i(A)\}_{i=1}^{n}$ are the eigenvalues of $A$.

\subsection{Behavioral system theory} 
We next recall some useful facts on behavioral system theory 
from~\cite{JCW-PR-IM-BDM:05}. 
For a sequence $\{w(k)\}_{k=0}^{T-1}$ with $w(k) \in \mathbb{R}^q$, we define the \textit{Hankel matrix of depth $L \in \mathbb{Z}_{\geq0}$} as

\scalebox{.9}{\parbox{\linewidth}{
\begin{equation*}
H_L(\{w(k)\}_{k=0}^{T-1}) = 
\begin{bmatrix} 
w(0) & w(1) & \cdots & w(T-L) \\ 
w(1) & w(2) & \cdots & w(T-L+1) \\ 
\vdots & \vdots & \ddots & \vdots \\ 
w(L-1) & w(L) & \cdots & w(T-1) 
\end{bmatrix},
\end{equation*}
}}
which satisfies $H_L(w(k)) \in \mathbb{R}^{qL \times (T-L+1)}$.

\begin{definition}[\bf \textit{Persistently exciting 
sequence~\cite{JCW-PR-IM-BDM:05}}]
The sequence $\{w(k)\}_{k=0}^{T-1}$, $w(k) \in \mathbb{R}^q$, is said to be
\textit{persistently exciting of order $L \in \mathbb{Z}_{\geq0}$} if 
$H_L(\{w(k)\})$ has rank $qL$ (i.e., is full row rank).\QEDB
\end{definition}
Note that, for a signal/sequence to be persistently exciting of order $L$, it must be sufficiently long; namely, 
$T \geq (q + 1)L - 1$.
\smallskip

Consider the linear dynamical system
\begin{align} \label{eq:auxDynamicalSystem}
    x(k+1) &= Ax(k)+Bu(k),
\end{align}
where $x(k) \in \mathbb{R}^n$, $u(k) \in \mathbb{R}^{m}$, 
$A \in \mathbb{R}^{n \times n}$, and $B \in \mathbb{R}^{n \times m}$.
We next recall the following property of~\eqref{eq:auxDynamicalSystem} when 
its inputs are persistently exciting. 

\begin{lemma}{\textit{\textbf{(Willem's Fundamental Lemma~\cite[Corollary 2]{JCW-PR-IM-BDM:05})}}}
\label{lem:fundLemmarankHankelMatrix}
Assume \eqref{eq:auxDynamicalSystem} is controllable, let 
$(\{u(k)\}_{k=0}^{T-1}, \{x(k)\}_{k=0}^{T-1})$ be an input-state trajectory of 
\eqref{eq:auxDynamicalSystem}, and $L \in \integnneg.$
If $\{u(k) \}_{k=0}^{T-1}$ is persistently exciting of order $n+L$, then:
\begin{align*}
\Rank
\begin{bmatrix} H_L (\{u(k)\}_{k=0}^{T-1}) \\  
H_1 (\{x(k)\}_{k=0}^{T-1}) \end{bmatrix} 
= L m + n.
\end{align*}
\QEDB\end{lemma}
\smallskip

\section{Problem formulation} \label{sec:problem}

We consider a \textit{true} physical system described by the discrete-time linear time-invariant model\footnote{In line with the system identification literature~\cite{LL:99}, we refer to the true physical system as the system that generates the data, which coincides with the system to which the controller is applied.}:
\begin{align} \label{eq:system}
    x(k+1) &= Ax(k)+Bu(k),
\end{align}
where $k \in \integnneg$ denotes time, $x(k) \in \mathbb{R}^n$ is the state vector, and $u(k) \in \mathbb{R}^{m}$ is the control input. The matrices $A \in \mathbb{R}^{n \times n}$ and $B \in \mathbb{R}^{n \times m}$ are such that the pair $(A, B)$ is 
controllable. 
In the remainder, we implicitly assume that the system is 
equipped with a sufficient number of sensors such that the 
full state vector is accessible for measurement; that is, 
the output equation can be simplified to $y(k) = x(k)$.

Following standard control design paradigms, we assume that the 
system operator (hereafter, the \textit{planner}) seeks to 
synthesize a controller matrix $K \in \mathbb{R}^{m \times n}$  and implements the state-feedback law
\begin{align}\label{eq:controller}
    u(k) = K x(k).
\end{align}
We assume that the planner has no knowledge of the 
system matrices $A$ and $B$, and therefore the synthesis must be 
carried out using exclusively experimental data generated 
by~\eqref{eq:system}.

We consider a setting where an adversary (hereafter, the 
\textit{attacker}) interferes with the controller design phase by 
poisoning the data used by the planner for synthesis. 
Similarly to the planner, the attacker has no knowledge of the system matrices $A$ and $B$. 
The two players\footnote{From a game-theoretic perspective, our setting is that of a \textit{Stackelberg game,} with the attacker being the leader and the planner being the follower.}  pursue opposing objectives: the planner seeks to design $K$ to stabilize the true closed-loop system, whereas the attacker aims to poison the data so that the resulting controller yields an unstable closed-loop system.
The planner-attacker interaction is sequential, and unfolds as 
follows (see Fig.~\ref{fig:attack_mechanism}):

\begin{enumerate}[leftmargin=*]
\item {\bf \textit{Identification experiment:}} The planner 
designs an input sequence $\{u(k)\}_{k=0}^{T-1}$ and applies it to~\eqref{eq:system}. The resulting state trajectory 
$\{x(k)\}_{k=0}^{T}$ is recorded,  transmitted, and stored, 
each potentially being a vulnerable medium.\label{step1}
\item {\bf \textit{Adversarial data poisoning:}} During 
recording,  transmission, or storing over the vulnerable 
medium, the attacker alters the recorded state\footnote{Although our 
framework could be extended to accommodate attacks that 
also alter the input sequence 
$\{u(k)\}_{k=0}^{T-1}$ into $\{\tilde{u}(k)\}_{k=0}^{T-1}$, 
we restrict attention to perturbations of the state 
trajectory only. This is because, since the input is designed by 
the planner in step~1), perturbations to the input sequence may more easily be detected.} trajectory
$\{x(k)\}_{k=0}^{T}$ into 
$\{\tilde{x}(k)\}_{k=0}^{T}$.
\item {\bf \textit{Controller synthesis and deployment:}} Unaware of the 
data alteration, the planner uses the poisoned dataset 
$(\{u(k)\}_{k=0}^{T-1}, \{\tilde{x}(k)\}_{k=0}^{T})$ to synthesize the feedback gain $K$ in~\eqref{eq:controller}. 
The controller is then deployed, 
yielding the closed-loop system
\eqref{eq:system}--\eqref{eq:controller}.
\end{enumerate}

We assume that the identification experiment is 
well-chosen, in the following sense. 

\begin{assumption}[\textbf{\textit{Persistently exciting experiments}}]
\label{as:pers_exc_input}
The input sequence $\{u(k)\}_{k=0}^{T-1}$ designed by the 
planner at step~\ref{step1}) is persistently exciting of 
order $n+1$.
\QEDB\end{assumption}

To avoid degenerate cases, we require that the 
dataset $\{\tilde{x}(k)\}_{k=0}^{T}$ generated by the 
attacker in step 2) is consistent with a controllable 
system, thereby preserving controllability of the\footnote{Note that, by~\cite{van2020data}, in our setting, there exists a unique system compatible with the dataset $(\{u(k)\}_{k=0}^{T-1}, \{\tilde{x}(k)\}_{k=0}^{T})$.} system compatible with the poisoned dataset.

It is worth emphasizing that we do not assume the attacker knows 
how the controller is synthesized in step 3). The only assumption 
is that the attacker is aware of the following.

\begin{assumption}[\bf \textit{Stability of the controller synthesis}]
\label{as:stability_synthesis}
At step 3), the controller matrix $K$ 
in~\eqref{eq:controller} is synthesized so as to 
stabilize the open-loop system consistent with 
the poisoned dataset 
$(\{u(k)\}_{k=0}^{T-1}, \{\tilde{x}(k)\}_{k=0}^{T})$.
\QEDB\end{assumption}

It is worth stressing that, under the considered attack model, the adversary does not directly interfere with the true 
system~\eqref{eq:system}, nor with its actuators or sensors during 
open-loop or closed-loop operation. Instead, the attack acts 
indirectly and offline: by poisoning the identification data, the 
attacker induces the planner to design a mismatched, destabilizing 
controller.
Motivated by the attacker's objectives, we introduce the following notion.

\begin{definition}[\bf \textit{Effective attack for destabilization}]
We say that an attack as in 1)--3) is \textit{effective for 
destabilization} if,
for any controller matrix $K$ designed in step~3) 
and 
satisfying Assumption~\ref{as:stability_synthesis}, the matrix $A+BK$ is not 
Schur stable. 
Conversely, we say that the attack is \textit{not effective for 
destabilization} if there exists $K$ satisfying 
Assumption~\ref{as:stability_synthesis} such that $A+BK$ is
Schur stable.
\QEDB\end{definition}

In words, an attack is effective for destabilization when, for any 
state-feedback law synthesized in step~3), the resulting closed-loop 
system~\eqref{eq:system}--\eqref{eq:controller} is not asymptotically stable.

We conclude this section by formalizing the problem of interest in this work.

\begin{problem}[\bf \textit{Objective of this work}]
\label{prob:main}
Given access to the data sequences $(\{u(k)\}_{k=0}^{T-1}, \{x(k)\}_{k=0}^{T})$, devise, when possible, a systematic procedure to construct a poisoned state trajectory $\{\tilde{x}(k)\}_{k=0}^{T}$ such that the resulting attack is effective for destabilization.~
\QEDB\end{problem}

We stress that Problem~\ref{prob:main} is challenging for two main reasons. 
First, the attacker does not know the true system~\eqref{eq:system} and must design the attack using only the data sequences $(\{u(k)\}_{k=0}^{T-1}, \{x(k)\}_{k=0}^{T})$. Second, the attacker has no knowledge of the controller synthesis procedure in step~3); instead, the attack must destabilize \textit{any} closed-loop system resulting from the application of any controller $K$ 
satisfying Assumption~\ref{as:stability_synthesis}. In this sense, the attack action must be robust against uncertainty in the controller synthesis procedure.

\section{Main results}
\label{sec:main_results}

In this section, we present two characterizations of 
destabilizing attacks: one formulated in terms of Hankel 
matrices, and the other of state trajectories.

\subsection{Sufficient conditions for effective attacks}
To state the results, we introduce the following compact notation. Given sequences $\{x(k)\}_{k=0}^{T}$ and $\{u(k)\}_{k=0}^{T-1}$, define
\begin{align*}
U &= H_1(\{u(k)\}_{k=0}^{T-1}) = \begin{bmatrix} u(0) & u(1) & \cdots & u(T-1) \end{bmatrix}, \\
X_{-} &= H_1(\{x(k)\}_{k=0}^{T-1}) = \begin{bmatrix} x(0) & x(1) & \cdots & x(T-1) \end{bmatrix}, \\
X_{+} &= H_1(\{x(k)\}_{k=1}^{T}) = \begin{bmatrix} x(1) & x(2) & \cdots & x(T) \end{bmatrix}.
\end{align*}
Analogous notation will be adopted for 
$\{\tilde x(k)\}_{k=0}^{T}.$

\begin{theorem}[\bf \textit{Data poisoning attack construction}]
\label{thm:x+shift}
Let Assumptions~\ref{as:pers_exc_input}--\ref{as:stability_synthesis} hold. 
Suppose the poisoned state trajectory 
is such that:
\begin{align}\label{eq:Xtilde_plus}
\tilde{X}_+ = X_+ + \delta X_-, &&
\tilde{X}_- = {X}_-,
\end{align}
for some $\delta \in \mathbb{R}$ satisfying $|\delta| \ge 2$. Then, the attack is effective for destabilization.
\QEDB\end{theorem}

\begin{proof}
Assume the planner successfully synthesizes a stabilizing controller from the poisoned dataset $(U, X_-, \tilde{X}_+)$. Under direct data-driven control frameworks~\cite[Theorem~2]{de2019formulas}, the state-feedback gain is parameterized by a decision matrix $G$ such that the apparent closed-loop matrix is characterized by 
\begin{align}\label{eq:appsysg}
    x(k+1) = \tilde{X}_+ Gx(k),  
\end{align}
where $G$ satisfies 
\begin{equation} \label{eq:g}
\begin{bmatrix} X_- \\ U \end{bmatrix} G = \begin{bmatrix} I \\ K \end{bmatrix}.
\end{equation}
Since Assumption~\ref{as:pers_exc_input} holds, Lemma~\ref{lem:fundLemmarankHankelMatrix} ensures that 
$\begin{bmatrix} X_- \\ U \end{bmatrix}$ has full row rank, 
and therefore the matrix $G$ satisfying~\eqref{eq:g} exists.
The planner ensures this apparent system~\eqref{eq:appsysg} is Schur stable. 
Thus, all eigenvalues $\tilde{\lambda}_i$ of $\tilde{X}_+ G$ satisfy $|\tilde{\lambda}_i| < 1$.

To evaluate the physical consequence of deploying the synthesized gain $K = U G$ from~\eqref{eq:g}, we determine the true data-driven closed-loop matrix. Because the attacker's specific perturbation in this attack strictly preserves the past data matrix $X_-$, the planner's structural synthesis constraint $X_- G = I$ remains valid for the true system~\eqref{eq:system}. Consequently, under the data-driven parameterization framework, the true physical closed-loop dynamics are
\begin{align}\label{eq:truesysg}
    x(k+1) = X_+ Gx(k).  
\end{align}
Substituting the attacker's algebraic perturbation~\eqref{eq:Xtilde_plus} into the parameterization of the apparent system\ \eqref{eq:appsysg}, and applying the constraint $X_- G = I$ from~\eqref{eq:g}, we obtain
\begin{align}\label{eq:data_shift}
    x(k+1) = (X_+ G + \delta I)x(k).  
\end{align}
Let $\lambda_i \in \mathbb{C}$ be an eigenvalue of the true data-driven closed-loop system~\eqref{eq:truesysg} with a corresponding eigenvector $v_i \neq 0$, such that $(X_+ G)v_i = \lambda_i v_i$. 
 Right-multiplying the system matrix in~\eqref{eq:data_shift} by $v_i$ yields:
\begin{align*}
     (X_+ G + \delta I) v_i =(\lambda_i + \delta)v_i.
\end{align*}
This establishes the exact spectral relation between the data matrices: the corresponding eigenvalue of the apparent system~\eqref{eq:data_shift} is exactly $\tilde{\lambda}_i = \lambda_i + \delta$. 
Rearranging this algebraic mapping for the true eigenvalue and applying the reverse triangle inequality provides a strict lower bound on its magnitude:
\begin{align*}
    |\lambda_i| &= |\tilde{\lambda}_i - \delta| \ge |\delta| - |\tilde{\lambda}_i|.
\end{align*}
Substituting the planner's Schur stability guarantee ($|\tilde{\lambda}_i| < 1$) into this inequality yields:
\begin{align*}
    |\lambda_i| > |\delta| - 1.
\end{align*}
Given the specified attack magnitude $|\delta| \ge 2$, the bounding inequality evaluates to $|\lambda_i| > 2 - 1 = 1$. Therefore, every eigenvalue of the true physical closed-loop system \eqref{eq:truesysg} resides outside the unit circle, rendering the true closed-loop system unstable. 
\end{proof}

Three important implications follow from Theorem~\ref{thm:x+shift}. 
First, it establishes that a destabilizing poisoning attack always exists, independently of the properties of~\eqref{eq:system}. Second, it 
shows that a poisoning attack can be designed without knowing the matrices $A$ and $B$ in~\eqref{eq:system}. Third, it demonstrates that the attack can be designed without requiring knowledge of the controller synthesis 
procedure, but only relying on the assumption that the controller is intended to be  stabilizing. 
Finally, we note that the condition $|\delta| \geq 2$ provides flexibility to the attacker, who can arbitrarily select $\delta$ within this range. We numerically investigate in Fig.~\ref{fig:delta-vs-noise} how the choice of $\delta$ interacts with measurement noise.

 \begin{remark}[\bf \textit{Model-based reinterpretation of Theorem~\ref{thm:x+shift}}]
 \label{rem:model-based_interpretation}
It follows from~\eqref{eq:truesysg}--\eqref{eq:data_shift}  
that, from the the planner perspective, the poisoned dataset appears to be generated by an \textit{apparent} system given by 
$x(k+1) = \tilde Ax(k)+ \tilde Bu(k)$, with $\tilde{A} = A + \delta I$ and $\tilde{B} = B$. 
\QEDB
\end{remark}

\subsection{Design of trajectory-compatible poisoned trajectories}

Although Theorem~\ref{thm:x+shift} provides a condition for constructing 
poisoned state trajectories, the representation~\eqref{eq:Xtilde_plus}, 
expressed in terms of Hankel matrices, does not necessarily correspond to a 
sequential trajectory $\{\tilde{x}(k)\}_{k=0}^{T}$. To clarify this point, 
express the poisoned trajectory as a perturbation of the true state trajectory:
\begin{align}\label{eq:trajectory-poisoned}
\tilde{x}(k) = x(k) + \Delta(k).
\end{align}
Define the associated Hankel matrices
\begin{align}\label{eq:data_cons}
\Delta_- &= \begin{bmatrix} \Delta(0) & \cdots & \Delta(T-1) \end{bmatrix}, \notag\\
\Delta_+ &= \begin{bmatrix} \Delta(1) & \cdots & \Delta(T) \end{bmatrix}.
\end{align}
For the poisoned data to be consistent with a trajectory $\{\tilde{x}(k)\}_{k=0}^{T}$, the Hankel matrices must satisfy
\begin{align} \label{eq:perturb-delta}
\tilde{X}_- = X_- + \Delta_-, \qquad 
\tilde{X}_+ = X_+ + \Delta_+,
\end{align}
which imposes a shift structure between $\tilde X_-$ and $\tilde X_+$ 
inherited from $\Delta(k)$.
By substitution, it is immediate to realize that 
construction~\eqref{eq:Xtilde_plus} provided in Theorem~\ref{thm:x+shift} does 
not satisfy~\eqref{eq:perturb-delta}. 
In the next subsection, we provide a technique to generate poisoning attacks 
that overcome this limitation. 

\begin{remark}{\bf \textit{(Practical relevance of Theorem~\ref{thm:x+shift})}}
Although it may not satisfy~\eqref{eq:data_cons}, the 
construction in Theorem~\ref{thm:x+shift} is applicable in scenarios where the attacker acts on the stored dataset, and data are represented in Hankel matrix form.
Moreover, \eqref{eq:Xtilde_plus} is also applicable in cases where data is not collected as a single, sequential trajectory, but rather as independent (one-step transitions).
\QEDB\end{remark}

Motivated by the requirements~\eqref{eq:trajectory-poisoned}--\eqref{eq:perturb-delta}, we introduce the following notion. 

\begin{definition}[\bf \textit{Trajectory-compatible attack}]
We say that an attack effective for destabilization is \textit{trajectory-compatible} if there exists a sequence $\{\Delta(k)\}_{k=0}^{T}$ such that
\begin{align*}
\tilde{X}_- = X_- + \Delta_-, \qquad 
\tilde{X}_+ = X_+ + \Delta_+,
\end{align*}
where $\Delta_-$, $\Delta_+$ are obtained from $\{\Delta(k)\}_{k=0}^{T}$ 
as in~\eqref{eq:data_cons}.~
\QEDB
\end{definition}

We present in Algorithm~\ref{alg:trajectory_spoofing} a procedure to systematically design trajectory-compatible attacks that are effective for destabilization. Intuitively, the algorithm computes a perturbation sequence that preserves trajectory consistency while injecting a destabilizing shift.
\begin{algorithm}[t]
\caption{Poisoned state trajectory generation}
\label{alg:trajectory_spoofing}
\begin{algorithmic}[1]
\Statex \textbf{Input:} Dataset $(\{{u}(k)\}_{k=0}^{T-1},\{{x}(k)\}_{k=0}^{T})$, desired perturbation size $\delta \in \mathbb{R}$

\State Construct the Hankel matrices
\begin{align*}
U &= H_1(\{u(k)\}_{k=0}^{T-1}), &
X_{-} &= H_1(\{x(k)\}_{k=0}^{T-1}), \\
X_{+} &= H_1(\{x(k)\}_{k=1}^{T});
\end{align*}

\State Set $\Delta(0) = 0$, and $\tilde{x}(0) = x(0)$;

\For{$k = 0, \dots, T-1$}
    \State Compute $g(k)$, a solution to: \label{eq:gk}
    \begin{align*}
        \begin{bmatrix} X_- \\ U \end{bmatrix} g(k) = \begin{bmatrix} \Delta(k) \\ \mathbf{0}  \end{bmatrix} 
    \end{align*} 
    \State Compute \label{eq:delta_k}
    \begin{align*}
        \Delta(k+1) = X_+ g(k) + \delta ( x(k) + \Delta(k) );
    \end{align*}
    \State Set
    \begin{align*}
        \tilde{x}(k+1) = x(k+1) + \Delta(k+1);
    \end{align*}
\EndFor
\State \Return $\{\tilde{x}(k)\}_{k=0}^{T}$
\end{algorithmic}
\end{algorithm}

The algorithm recursively constructs a poisoned trajectory $\{\tilde x(k)\}_{k=0}^T$ as follows: in line~\ref{eq:delta_k}, it combines a data-consistent propagation term, obtained from the recorded trajectory through the coefficient $g(k)$ (computed in line~\ref{eq:gk}), with an additive bias proportional to the current poisoned state. The coefficient $g(k)$ is chosen so that the current perturbation $\Delta(k)$ is represented using the measured state-input data without modifying the input sequence. In line~\ref{eq:delta_k}, the term $\delta ( x(k) + \Delta(k) )= \delta \tilde x(k)$ injects the shift needed to enforce destabilization. Finally, the choice $\Delta(0)=0$ ensures consistency of the initial state, yielding $\tilde x(0)=x(0)$.

The following result guarantees that attacks generated by
Algorithm~\ref{alg:trajectory_spoofing} are effective for destabilization.

\begin{theorem}[\bf \textit{Correctness of Algorithm~\ref{alg:trajectory_spoofing}}]
\label{thm:data_driven_destabilization}
Let Assumptions~\ref{as:pers_exc_input}--\ref{as:stability_synthesis} hold. 
Suppose the poisoned state trajectory $\{\tilde{x}(k)\}_{k=0}^{T}$ 
is generated by Algorithm~\ref{alg:trajectory_spoofing} with 
$|\delta| \ge 2$. 
Then, the attack is effective for destabilization.
\QEDB\end{theorem}

\begin{proof}
Since Assumption~\ref{as:pers_exc_input} holds, Lemma~\ref{lem:fundLemmarankHankelMatrix} ensures that 
$\begin{bmatrix} X_- \\ U \end{bmatrix}$ has full row rank. 
Therefore, the linear system in line~\ref{eq:gk} admits a solution $g(k)$ 
for each $k \in [0,~ T-1]$, and by concatenating these solutions, we obtain $G_\Delta = \begin{bmatrix} g(0) & g(1) & \dots & g(T-1) \end{bmatrix}$.
By construction, this matrix satisfies the data equation:
    \begin{equation}\label{eq:g-delta}
        \begin{bmatrix} X_- \\ U \end{bmatrix} G_\Delta = \begin{bmatrix} \Delta_- \\ \mathbf{0} \end{bmatrix}.
    \end{equation}
    Define the transformation matrix $G_{map} = I + G_\Delta$. Applying this mapping to the true stacked data and using~\eqref{eq:g-delta}, we obtain:
    \begin{equation} \label{eq:g_map}
        \begin{bmatrix} X_- \\ U \end{bmatrix} G_{map} = \begin{bmatrix} X_- + \Delta_- \\ U \end{bmatrix} = \begin{bmatrix} \tilde{X}_- \\ U \end{bmatrix}.
    \end{equation}
Similarly, evaluating the recursive update rule on line~\ref{eq:delta_k} for $T$ steps to yield $\Delta_+ = X_+ G_\Delta + \delta \Delta_- + \delta X_-$. Substituting this into $\tilde{X}_+ = X_+ + \Delta_+$ gives:
    \begin{equation}\label{eq:open_loop_alg}
       \tilde{X}_+ = X_+ G_{map} + \delta \tilde{X}_-.
    \end{equation}
By Proposition~\ref{prop:model_rank} in the Appendix, the 
dataset $(U,\tilde X_-,\tilde X_+)$ is a trajectory of the 
apparent system $x(k+1)=(A+\delta I)x(k)+Bu(k).$
Since $(A,B)$ is controllable, the pair $(A+\delta I,B)$ is
controllable for every $\delta\in\mathbb R$. Moreover, by
Assumption~\ref{as:pers_exc_input}, the input sequence 
$\{u(k)\}_{k=0}^{T-1}$ is persistently exciting of order 
$n+1$, and thus Lemma~\ref{lem:fundLemmarankHankelMatrix} guarantees:
\begin{equation*}
\operatorname{Rank} \begin{bmatrix} \tilde{X}_- \\ U \end{bmatrix}  = n+m.
\end{equation*}
Now, assume the planner synthesizes a stabilizing state feedback 
gain $K$ from the poisoned dataset $(U, \tilde{X}_-, \tilde{X}_+)$.
By~\cite[Theorem~2]{de2019formulas}, the state-feedback 
gain can parameterized by a decision matrix $\tilde{G}$ such that the apparent closed-
loop matrix is characterized by
\begin{equation}\label{eq:system-g-tilde}
    x(k+1) = \tilde{X}_+ \tilde{G} x(k),
\end{equation} 
where the matrix $\tilde{G}$  satisfies
\begin{equation}\label{eq:g-tilde}
    \begin{bmatrix} \tilde{X}_- \\ U \end{bmatrix} \tilde{G} = \begin{bmatrix} I \\ K \end{bmatrix}.
\end{equation}   Substituting~\eqref{eq:open_loop_alg} into the planner's closed-loop system~\eqref{eq:system-g-tilde} and applying  $\tilde{X}_- \tilde{G} = I$ from~\eqref{eq:g-tilde} simplifies the apparent dynamics to:
\begin{equation}\label{eq:shifted-system}
         x(k+1) = (X_+ G_{map} \tilde{G} + \delta I) x(k).
    \end{equation}
We now analyze the composite operator $G_{true} = G_{map} \tilde{G}$ applied to the true dataset. Utilizing~\eqref{eq:g_map} and~\eqref{eq:g-tilde}, we obtain:
    \begin{equation} \label{eq:g-true}
        \begin{bmatrix} X_- \\ U \end{bmatrix} G_{true} = \left( \begin{bmatrix} X_- \\ U \end{bmatrix} G_{map} \right) \tilde{G} = \begin{bmatrix} \tilde{X}_- \\ U \end{bmatrix} \tilde{G} = \begin{bmatrix} I \\ K \end{bmatrix}.
    \end{equation}
Because $G_{true}$ satisfies the fundamental parameterization constraints for the unperturbed dataset, the true physical closed-loop dynamics are exactly governed by 
    \begin{equation}\label{eq:gtrue-system}
        x(k+1) = X_+ G_{true} x(k),
    \end{equation}
    where $G_{true}$ satisfies~\eqref{eq:g-true}. 
Substituting $G_{true}$ into~\eqref{eq:shifted-system} establishes the exact relation:    \begin{equation}\label{eq:theorem2-relation}
        \tilde{X}_+ \tilde{G} = X_+ G_{true} + \delta I.
    \end{equation}
    Let $\lambda_i \in \mathbb{C}$ be an eigenvalue of the true data-driven closed-loop system matrix $X_+ G_{true}$ as in~\eqref{eq:gtrue-system} with corresponding eigenvector $v_i \neq 0$. 
    Right-multiplying the system matrix in~\eqref{eq:theorem2-relation} by $v_i$ yields:
\begin{align*}
    (\tilde{X}_+ \tilde{G}) v_i = (\lambda_i + \delta)v_i.
\end{align*}
This algebraic mapping dictates that $v_i$ is also an eigenvector of the apparent closed-loop system matrix $\tilde{X}_+ \tilde{G}$, strictly defining its corresponding eigenvalue as $\tilde{\lambda}_i = \lambda_i + \delta$. Following the analogous bounds established in Theorem~\ref{thm:x+shift}, selecting an attack magnitude of $|\delta| \ge 2$ guarantees that the true physical closed-loop system is destabilized.
\end{proof}
As noted, Theorem~\ref{thm:data_driven_destabilization} states that 
Algorithm~\ref{alg:trajectory_spoofing} generates a poisoned trajectory that is consistent with a valid state evolution and guarantees destabilization of the closed-loop system for any admissible controller when $|\delta|\ge 2$.
This result extends the construction in Theorem~\ref{thm:x+shift} to ensure compatibility with a valid sequential poisoned-state trajectory.
\begin{remark}{\bf \textit{(Hankel matrix interpretation of Algorithm~\ref{alg:trajectory_spoofing})}}
In terms of Hankel matrices, the attack generated by Algorithm~\ref{alg:trajectory_spoofing} can be written as (see~\eqref{eq:g_map}--\eqref{eq:open_loop_alg})
\begin{align*}
\tilde X_+ = (X_+ + \delta X_-) G_{{map}}, \qquad
\tilde X_- = X_- G_{{map}},
\end{align*}
for the matrix $G_{{map}}$ defined in~\eqref{eq:g_map}. 
Comparing with~\eqref{eq:Xtilde_plus}, it follows that Algorithm~\ref{alg:trajectory_spoofing} exploits additional degrees of freedom in the construction of $\tilde X_+$ and $\tilde X_-$ through the choice of $G_{{map}}$.
\QEDB
\end{remark}

\begin{remark}[\bf \textit{Model-based reinterpretation of Theorem~\ref{thm:data_driven_destabilization}}]\label{rem:model-based_interpretation_v2}
It follows from~\eqref{eq:theorem2-relation} that, to the planner, the perturbed data appears to be generated by an \textit{apparent} system given by 
$x(k+1) = \tilde Ax(k)+ \tilde Bu(k)$ with $\tilde{A} = A + \delta I$ and $\tilde{B} = B$. 
By comparing this
outcome with Remark~\ref{rem:model-based_interpretation}, it is evident that both the direct
algebraic matrix perturbation in Theorem~\ref{thm:x+shift} and the recursive trajectory generation in Theorem~\ref{thm:data_driven_destabilization} induce the exact same parametric shift.
\end{remark}


\subsection{Minimum-magnitude effective attacks for destabilization}
Although the construction in Theorem~\ref{thm:x+shift} is 
sufficient to ensure effectiveness for destabilization, the condition $|\delta| \geq 2$ is also necessary, as established next.

\begin{proposition}[\textbf{\textit{Insufficiency of threshold perturbations}}]\label{prop:delta<2}
Let Assumptions~\ref{as:pers_exc_input}--\ref{as:stability_synthesis} hold. 
Suppose the attacker generates a poisoned dataset such that $\tilde{X}_+ = X_+ + \delta X_-$ and $\tilde{X}_- = {X}_-$ 
for some magnitude $\delta \in \mathbb{R}$ satisfying $|\delta| < 2$. 
Then, the attack is not effective for destabilization. 
\end{proposition}

\begin{proof}
Assume the attacker selects a perturbation shift $|\delta| < 2$. Define the real set $\Lambda$ as the intersection of two open intervals:
$$ \Lambda := (-1, 1) \cap (-1 - \delta, 1 - \delta). $$
Because $|\delta| < 2$, the intersection $\Lambda$ is strictly nonempty. Let $\lambda^*$ be an arbitrary scalar chosen such that $\lambda^* \in \Lambda$. By the definition of this intersection, the following strict inequalities hold simultaneously:
$$ |\lambda^*| < 1 \quad \text{and} \quad |\lambda^* + \delta| < 1. $$

Since the dataset $(U, X_-, \tilde{X}_+)$ is persistently exciting and the underlying system is controllable, the data-driven parameterization in~\cite[Theorem~2]{de2019formulas} allows arbitrary closed-loop pole assignment for~\eqref{eq:appsysg} through a suitable choice of $G$.
Therefore, there exists a valid decision matrix $G$ that places the spectrum of the apparent closed-loop system entirely at $\lambda^* + \delta$, such that 
$\sigma(\tilde{X}_+G) = \{\lambda^* + \delta, \dots, \lambda^* + \delta\}$.

Because $|\lambda^* + \delta| < 1$, this synthesis matrix $\tilde{X}_+G$ strictly satisfies the planner's Schur stability.

Applying the exact spectral mapping ($\tilde{X}_+G = X_+G +\delta I$) established in Theorem~\ref{thm:x+shift} , the corresponding spectrum of the true physical closed-loop system is strictly evaluated as $\sigma(X_+G) = \{\lambda^*, \dots, \lambda^*\}$. 
Hence, the true closed-loop matrix $X_+G$ is Schur stable.
\end{proof}

One important consequence follows from Proposition~\ref{prop:delta<2}: by 
interpreting $\delta$ as a measure of the perturbation magnitude
in~\eqref{eq:Xtilde_plus}, the minimum value required to induce destabilization 
is $|\delta|=2$.

\begin{remark}[\bf \textit{Applicability to trajectory-consistent attack}]
The insufficiency bounds derived in Proposition~\ref{prop:delta<2} extend to the trajectory poisoning results of Theorem~\ref{thm:data_driven_destabilization}. Since the recursive trajectory poisoning algorithm is proven to induce the exact isotropic shift ($\tilde{A} = A + \delta I$) as in Remark~\ref{rem:model-based_interpretation_v2}, the attacker is strictly bound by the identical operational threshold.
\QEDB\end{remark}

\begin{figure}
    \centering
    \includegraphics[width=1.0\linewidth]{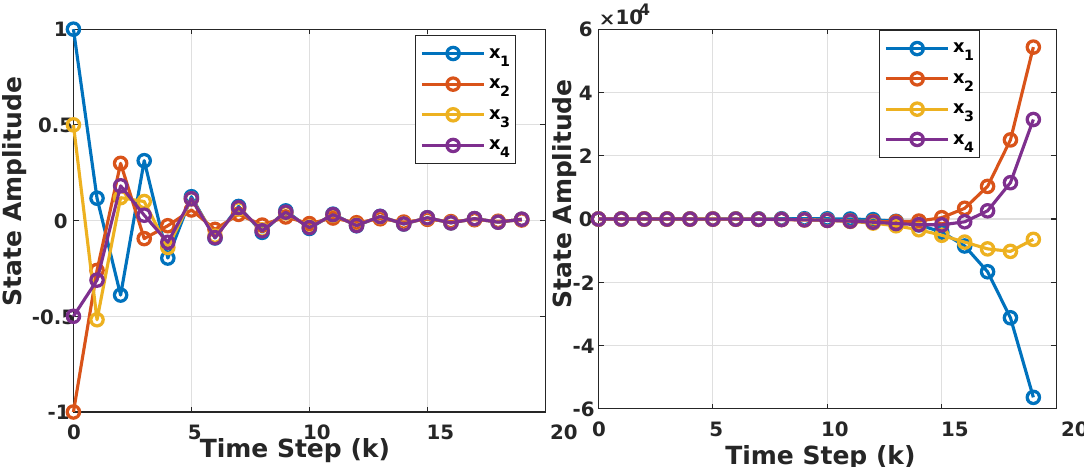}
    \caption{\small Closed-loop response of the system under the proposed data-poisoning attack. (Left) The apparent system exhibits asymptotic convergence, satisfying the planner's Schur stability condition. (Right) The true physical system diverges exponentially.}
    \label{fig:state-trajectory1}
\end{figure}

\begin{figure}
    \centering
    \includegraphics[scale = 0.85]{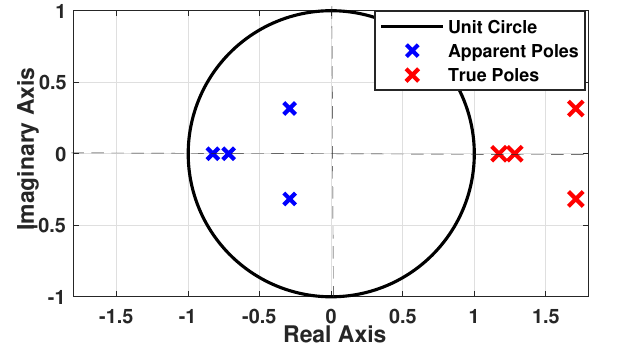}
    \caption{\small Closed-loop pole locations in the complex plane under the proposed data-poisoning attack; while the planner successfully constrains the apparent poles ($\tilde{\lambda}$) within the unit circle, while the true closed-loop poles ($\lambda$) lie outside the unit circle}
    \label{fig:theorem1-pole}
    \vspace{-0.5cm}
\end{figure}

\begin{remark}[\bf \textit{Anisotropic perturbation}]
While the uniform scalar perturbation guarantees destabilization against any data-driven controller, the resulting isotropic spectral shift requires a relatively large perturbation magnitude (e.g., $|\delta| \ge 2$).
This occurs because the isotropic shift uniformly displaces all eigenvalues of the closed-loop dynamics (see Fig.~\ref{fig:theorem1-pole}). 
A potential direction for future research is anisotropic data poisoning.
 If an attacker can estimate the dominant eigenvectors of the anticipated closed-loop system, they could selectively project perturbations along directions closest to the stability boundary. Such targeted perturbations may achieve destabilization with significantly smaller perturbation magnitude.
    \QEDB
\end{remark}

\section{Numerical simulations}\label{sec:simulation}
In this section, numerical simulations are provided to demonstrate the efficacy of the trajectory poisoning attack against a purely data-driven control synthesis. 
The evaluation utilizes the discrete-time batch reactor model considered in~\cite{de2019formulas}. The true underlying system matrices $[A|B] $, which remain strictly unknown to both the planner and the attacker, are specified as
\scalebox{.95}{\parbox{\linewidth}{
\begin{align*}
\begin{bmatrix} 1.178 & 0.001 & 0.511 & -0.403 & \big| & 0.004 & -0.087 \\ -0.051 & 0.661 & -0.011 & 0.061 & \big| & 0.467 & 0.001 \\ 0.076 & 0.335 & 0.560 & 0.382 & \big| & 0.213 & -0.235 \\ 0 & 0.335 & 0.089 & 0.849 & \big| & 0.213 & -0.016 \end{bmatrix}
\end{align*}}}.

We generate the states data by injecting a random white noise input signal $u(t)$ (sampled uniformly from $[-1, 1]$) with length $T=15$. Here we consider the planner algorithm to synthesizes a state-feedback gain $K$ of the form~\eqref{eq:controller} using the data-driven stabilization method presented in~\cite[Theorem~3]{de2019formulas}. We use CVX~\cite{grant2008cvx} to solve the optimization problem in MATLAB. 

\subsection{Simulations with noise-free data}
In the first scenario, the attacker perturbs only the $X_+$ sequence as defined in Theorem~\ref{thm:x+shift} with $\delta = -2.0$. In the second scenario (recursive perturbations), the attacker executes Algorithm~\ref{alg:trajectory_spoofing} with $\delta = -2.0$.
Because both methodologies strictly embed the identical shifted linear operator into the data matrices as presented in Remark~\ref{rem:model-based_interpretation_v2}, the LMI solver processes an identical apparent dynamic subspace. Consequently, the synthesis outcomes are indistinguishable. For both attack vectors, the LMI solver returns a strictly feasible status, confirming that the apparent closed-loop system is Schur stable. This fulfills the attacker's stealth requirement, as the planner is completely unaware of the malicious injection. 
However, when this synthesized gain is deployed to the system~\eqref{eq:system}, the true closed-loop state trajectories exponentially diverge, as shown in Fig.~\ref{fig:state-trajectory1}. Spectral analysis confirms that the true closed-loop system possesses all the eigenvalue strictly outside the unit circle as shown in Fig.~\ref{fig:theorem1-pole}, while the apparent closed-loop system poles lie within the unit circle.

To quantify the perturbation energy injected during the offline data-poisoning phase of Algorithm~\ref{alg:trajectory_spoofing}, Fig.~\ref{fig:trajectory_norms} evaluates the trajectory magnitudes. 
As illustrated in the left panel, the $L_2$ norm of the poisoned trajectory $\tilde{X}$ diverges significantly from the true physical trajectory $X$ after few time-steps. The mechanics driving this divergence are characterized in the right panel, which plots the perturbation norm $\|\Delta(k)\|_2$ on a logarithmic scale. 
This growth is dictated by the recursive injection dynamics, specifically the spectral radius of the shifted operator, $\rho(A + \delta I)$ as in Remark~\ref{rem:model-based_interpretation_v2}. 
Because the attack parameter $\delta$ is synthesized such that $\rho(A + \delta I) > 1$, the perturbation inherently simulates an unstable virtual mode, forcing the data matrices to expand exponentially before being processed by the planner.
\begin{figure}
    \centering
    \includegraphics[width=1.0\linewidth]{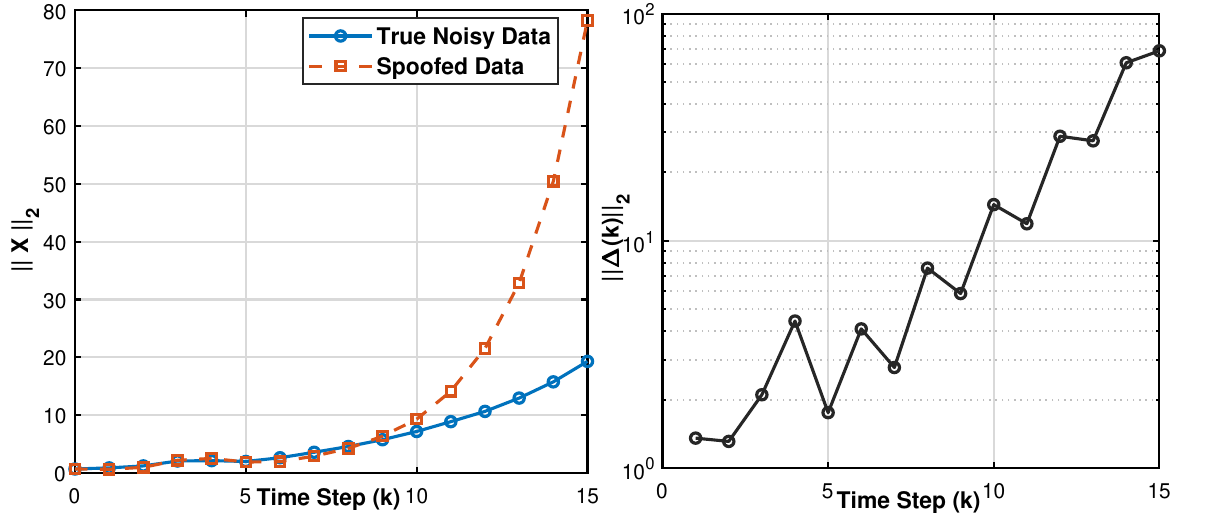}
    \caption{\small Offline data magnitude evaluation characterizing the attack's recursive injection magnitude $\delta = -2.0$ using Algorithm~\ref{alg:trajectory_spoofing}. (Left) Comparison of the $L_2$ norms of the true open-loop trajectory $X$ and the poisoned trajectory $\tilde{X}$. (Right) Logarithmic plot of the perturbation vector magnitude $\|\Delta(k)\|_2$, demonstrating the exponential growth governed by $\rho(A + \delta I)$.}
    \vspace{-0.4cm}
\label{fig:trajectory_norms}
\end{figure}

\subsection{Simulations with noisy data}

To assess the robustness of the attack under more realistic operational conditions, a random uniform measurement noise sequence $w(k)$ is injected during the offline data collection phase to system~\eqref{eq:system}.
We first evaluate the system under a moderate noise bound of $\|w(k)\|_\infty \le 0.05$ (sampled uniformly from $[-0.05, 0.05]$). Under this condition, the deterministic spectral shift injected by the attacker ($|\delta| = -2.0$) dominates the stochastic noise. The entire closed-loop spectrum is uniformly displaced, forcing all eigenvalues of the true physical system~\eqref{eq:system} to be strictly outside the unit circle, as shown in Fig.~\ref{fig:theorem1-pole-noise}.

\begin{figure}
    \centering
    \includegraphics[width=1.0\linewidth]{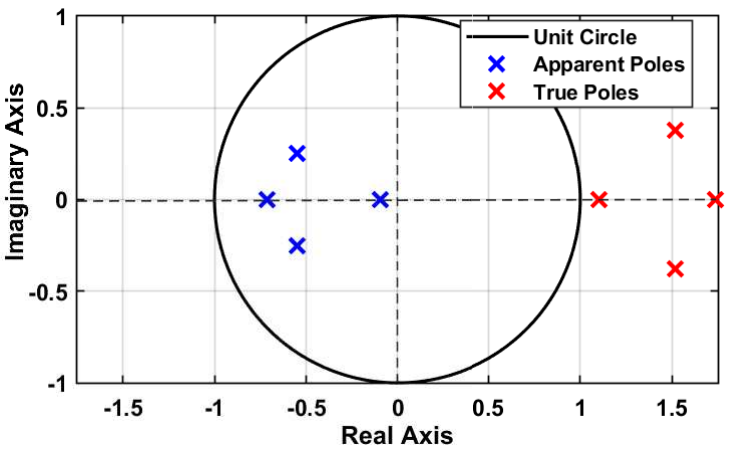}
    \caption{\small Closed-loop pole locations in the complex plane under the proposed data-poisoning attack; while the planner successfully constrains the apparent poles ($\tilde{\lambda}$) within the unit circle, while the true closed-loop poles ($\lambda$) lie outside the unit circle}
    \label{fig:theorem1-pole-noise}
\end{figure}
\begin{figure}
    \centering
     \vspace{-0.15cm}
     \includegraphics[width=1.0\linewidth]{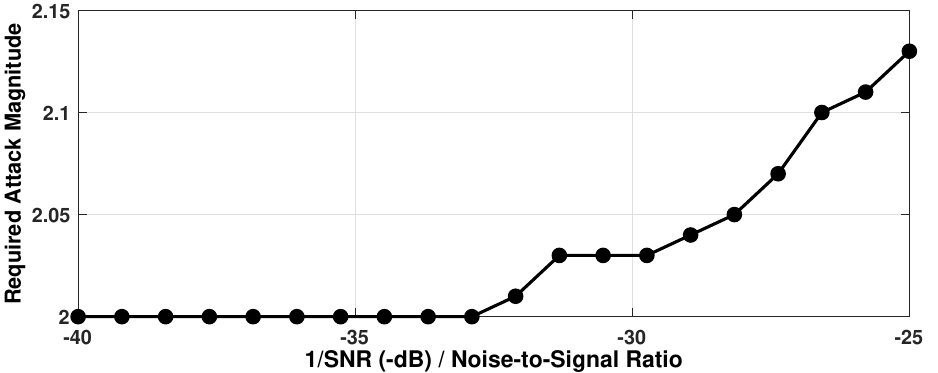}
    \caption{\small Scaling of the attack magnitude $|\delta|$ required to satisfy $\min_i |\lambda_i| > 1$ under increasing measurement noise in the system.}
    \label{fig:delta-vs-noise}
    \vspace{-0.5cm}
\end{figure}

To quantify the level of data poisoning, we evaluate the system's vulnerability across a spectrum of Noise-to-Signal Ratios (NSR), expressed in decibels in Fig.~\ref{fig:delta-vs-noise}. Mathematically, the NSR in decibels is the exact additive inverse of the Signal-to-Noise Ratio ($-\text{SNR}_{\text{dB}}$). This metric is formally defined as
 $\text{NSR}_{\text{dB}} := 20\log_{10}\left(\frac{\alpha\|W\|_F}{\|X\|_F}\right) $,
where $\alpha$ is the noise scaling parameter, $W$ is the matrix of additive simulation noise, and $X$ is the nominal state trajectory matrix. This formulation directly computes the logarithmic ratio of the noise matrix magnitude to the state trajectory magnitude, yielding the relative energy of the perturbation. Fig.~\ref{fig:delta-vs-noise} confirms that increasing the attack magnitude $\delta$ forces all true closed-loop poles to be strictly outside the unit disk.


\section{Conclusions}
In this paper, we revealed a critical weakness in 
data-driven control synthesis, showing that an attacker 
can systematically design offline data poisoning 
actions that cause any linear state feedback controller synthesized by the 
planner to destabilize the physical system.
We have explored two strategies for data- 
poisoning: one based on Hankel matrices, and an 
alternative recursive mechanism.
The proposed results motivate several directions for future research, 
including further restricting the attacker’s 
capabilities, deriving necessary conditions for data-
poisoning, and designing more advanced attack 
strategies under additional structural assumptions on 
the controller synthesis procedure.

                                  %




\bibliographystyle{IEEEtran}
\bibliography{BIB/alias,BIB/full_GB,BIB/GB}

\section{Appendix}

\begin{proposition}{\bf \textit{(Model-based reinterpretation of poisoned data from Algorithm~\ref{alg:trajectory_spoofing})}}\label{prop:model_rank}
Suppose the poisoned trajectory matrices $(\tilde{X}_-, \tilde{X}_+)$ are generated by Algorithm~\ref{alg:trajectory_spoofing}. Then, the poisoned dataset emulates the state trajectory of an apparent system given by 
$x(k+1) = \tilde Ax(k)+ \tilde Bu(k)$ with $\tilde{A} = A + \delta I$ and $\tilde{B} = B$. 
\end{proposition}


\begin{proof}
The true data matrices $(U, X_-, X_+)$ strictly satisfy the underlying physical dynamics:
$$X_+ = A X_- + B U.$$
Substituting this into the structural mapping~\eqref{eq:open_loop_alg} yields:
\begin{align*}
    \tilde{X}_+ = A(X_- G_{map}) + B(U G_{map}) + \delta \tilde{X}_-.
\end{align*}
Applying the block matrix mapping identities derived in~\eqref{eq:g_map}, specifically $X_- G_{map} = \tilde{X}_-$ and $U G_{map} = U$, the equation reduces to:
\begin{align*}
    \tilde{X}_+ = (A + \delta I)\tilde{X}_- + B U.
\end{align*}
Thus, the data matrices satisfy the apparent system parameterized by $\tilde{A} = A + \delta I$ and $\tilde{B} = B$.
\end{proof}

\end{document}